\documentclass{amsproc}
\usepackage{epsfig,amscd,amssymb,amsmath,amsfonts}
\usepackage[margin=1in]{geometry}
\newtheorem{theorem}{Theorem}[section]
\newtheorem{lemma}[theorem]{Lemma}

\theoremstyle{definition}
\newtheorem{definition}[theorem]{Definition}

\newtheorem{example}[theorem]{Example}

\theoremstyle{remark}
\newtheorem{remark}[theorem]{Remark}

\numberwithin{equation}{section}

\begin{document}

\title{$G$-Algebra Structure on the Higher Order Hochschild Cohomology $H^*_{S^2}(A,A)$}

\author{Samuel Carolus}
\address{Department of Mathematics and Statistics, Bowling Green State University, Bowling Green, OH 43403}
\email{carolus@bgsu.edu}

\author{Mihai D. Staic}
\address{Department of Mathematics and Statistics, Bowling Green State University, Bowling Green, OH 43403 }
\address{Institute of Mathematics of the Romanian Academy, PO.BOX 1-764, RO-70700 Bu\-cha\-rest, Romania.}
\email{mstaic@bgsu.edu}






\begin{abstract} 
We present a deformation theory associated to the higher Hochschild cohomology 
$H_{S^2}^*(A,A)$. We also study a $G$-algebra structure associated to this deformation theory.  
\end{abstract}

\maketitle

 \section{Introduction}

$G$-Algebra structures were introduced by Gerstenhaber in \cite{g1}, where he proved that on the Hochschild cohomology $H^*(A,A)$ one can define a cup product that is  graded commutative and a bracket that satisfies a graded Jacobi identity. Moreover there is a compatibility between the two structures, i.e. the bracket satisfies a graded Leibniz identity with respect to the cup product. Later it was proved in \cite{gv} that the existence of the $G$-algebra structure on $H^*(A,A)$ is a consequence of the existence of a multiplicative operad structure on $C^*(A,A)$. $G$-Algebra structures have been proven to exist in other settings like the exterior algebra of a Lie algebra, differential forms on a Poisson manifold, secondary Hochschild cohomology, etc. 

Higher order Hochschild homology of a commutative algebra $A$ over a simplicial set $X$ was introduced by Pirashvili in \cite{p}; the cohomology version was defined by Ginot in \cite{gi}. One interesting fact is that this (co)homology theory depends only on the homotopy type  of the simplicial set $X$. When $X$ is the simplicial set associated to the sphere $S^1$, one recovers the usual Hochschild cohomology. 

This paper was initially motivated by a question of Bruce Corrigan regarding the relevance of  higher Hochschild cohomology in deformation theory. We prove that the higher order Hochschild cohomology $H^2_{S^2}(A,A)$ controls deformations of morphisms $u:A[[t]]\to A[[t]]$ that satisfy the identity $u(ab)u(c)=u(a)u(bc)$. As usual in these situations, there is an obstruction which is an element in $H^3_{S^2}(A,A)$. A consequence of this deformation theory is the existence of a $G$-algebra structure on $H^*_{S^2}(A,A)$. 
We actually show the existence of a homotopy $G$-algebra structure on $C^*_{S^2}(A,A)$ which induces the cup product and bracket on $H^*_{S^2}(A,A)$. For this we use the language of  multiplicative operads from \cite{gv}, and the description of $H^*_{S^2}(A,M)$ given in \cite{l2}. 

One should notice that in \cite{gi} it was proved that $H_{S^d}(A,A)$ admits $d+1$-algebra structure. It is not clear if there exists a direct connection between that construction and the results from this paper.

\section{Preliminary} In this paper $k$ is a field and $A$ is a commutative $k$-algebra; we denote $\otimes_k$ by $\otimes$. 
\subsection{Higher order Hochschild cohomology}
 We recall from  \cite{gi} the construction of the Higher Hochschild cohomology.  Let $A$ be a commutative $k$-algebra and $M$ a symmetric $A$-bimodule.

Let $V$ be a finite pointed set such that $\vert V\vert=v+1$. We define $\mathcal{H}(A,M)(V)=Hom_k(A^{\otimes v},M)$. 
For $\phi:V\to W$   we define 
$$\mathcal{H}(A,M)(\phi): \mathcal{H}(A,M)(W)\to \mathcal{H}(A,M)(V)$$
determined as follows: if  $f\in \mathcal{H}(A,M)(V)$
then $$\mathcal{H}(A,M)(\phi)(f)(a_1\otimes ...\otimes a_v)=b_0f(b_1\otimes ...\otimes b_{w})$$ 
where $$b_i=\prod_{\{j\in V | j\neq *, \phi (j)=i\}}a_j.$$
Take $X_{\bullet}$ to be a finite pointed simplicial set.  
$$C_{X_\bullet}^n=\mathcal{H}(A,M)(X_n)$$ For each $d_i:X_{n+1}\to X_{n}$ we define $d_i^*=\mathcal{H}(A,M)(d_i):C_{X_{\bullet}}^{n}\to C_{X_{\bullet}}^{n+1}$ and take $\partial_{n}:C^{n}_{X_{\bullet}}\to C^{n+1}_{X_{\bullet}}$ defined as $\partial_n=\sum_{i=0}^{n+1}(-1)^i(d_i)^*$. 

The homology of this complex is denoted by $H^n_{X_{\bullet}}(A,M)$ and is called the higher order Hochschild cohomology group. One  interesting fact is that these groups depend only on the homotopy type of the geometric realization of the simplicial set $X_{\bullet}$.  When $X=S^1$ with the usual simplicial structure one recovers the complex that defines Hochschild cohomology.

 Next we recall from \cite{l2} (see also \cite{cl}) a description of the above complex when  $X=S^2$ with the following simplicial structure. Take $X_n=\{*_n\}\cup \{^a\Delta^b_c\vert a,b,c\in \mathbb{N}, a+b+c=n-2\}$. Define $d_i:X_n\to X_{n-1}$ determined by $d_i(*_n)=*_{n-1}$, $s_i(*_n)=*_{n+1}$ and 
\begin{eqnarray}
d_i(\,^a\Delta^b_c)= \left\{\begin{array}{ll}
  *_{a+b+c+1}& \mbox{ if $a=0$ and $i=0$}\\ 
  ^{a-1}\Delta^b_c & \mbox{ if  $a\neq 0$ and $i\leq a$ }\\
 *_{a+b+c+1}& \mbox{ if $b=0$ and $i=a+1$}\\
 ^{a}\Delta^{b-1}_c & \mbox{ if  $b\neq 0$ and $a<i\leq a+b+1$ }\\
 *_{a+b+c+1}& \mbox{ if $c=0$ and $i=n=a+b+2$}\\
 ^{a}\Delta^b_{c-1} & \mbox{ if  $c\neq 0$ and $i\geq a+b+2$}.\\
 \end{array}\right.\label{f3}
 \end{eqnarray}

\begin{eqnarray}
s_i(\,^a\Delta^b_c)= \left\{\begin{array}{ll} 
  ^{a+1}\Delta^b_c & \mbox{ if  $i\leq a$}\\
 ^{a}\Delta^{b+1}_c & \mbox{ if  $a+1\leq i\leq a+b+1$}\\
 ^{a}\Delta^b_{c+1} & \mbox{ if  $a+b+2\leq i$.}\\
 \end{array}\right.\label{f4}
 \end{eqnarray}

This gives $C_{S^2}^n(A,A)=Hom_k(A^{\otimes\frac{n(n-1)}{2}},A)$.  Just like in \cite{sta} it is convenient to have a tensor matrix notation for an element in $A^{\otimes \frac{s(s-1)}{2}}$. We will consider a $s\times s$  upper triangular matrix that has $1$'s on the diagonal and elements in $A$ above the diagonal. The $1$'s on the diagonal do not play any role except for making our notation more intuitive. 
 \begin{eqnarray}
T=\otimes \begin{pmatrix} 1 & a_{1,2} & a_{1,3}  & \cdots&a_{1,s-1} & a_{1,s}\\ 
& 1 & a_{2,2}  & \cdots &a_{2,s-1}& a_{2,s}\\ 
&  & 1  & \cdots&a_{3,s-1}& a_{2,s}\\ 
& & & \ddots & \vdots & \vdots\\ 
& & &&1& a_{s-1,s}\\ 
& & & && 1 \end{pmatrix}\in A^{\otimes \frac{s(s-1)}{2}}.
\label{matrixT}
\end{eqnarray}
If we identify $^a\Delta^b_c$ with the position $(a+1,a+b+2)$ in the tensor matrix then the differential map $\delta_n^{S^2}:C^n_{S^2}(A,A)\to C^{n+1}_{S^2}(A,A)$ becomes: 
\begin{eqnarray}
&\delta_n^{S^2}(f)(\otimes \begin{pmatrix} 1 & a_{1,2}   & \cdots &a_{1,n}& a_{1,n+1}\\ 
 & 1   & \cdots &a_{2,n}& a_{2,n+1}\\ 
& &  \ddots & \vdots&\vdots\\ 
&  & &  1&a_{n,n+1}\\ 
&  & & & 1 \end{pmatrix})=f(\otimes \begin{pmatrix} 1 & a_{2,3}   & \cdots &a_{2,n}& a_{2,n+1}\\ 
& 1  & \cdots &a_{3,n}& a_{3,n+1}\\
& &  \ddots & \vdots&\vdots\\ 
&  & & 1&a_{n,n+1}\\ 
&  & & &1 \end{pmatrix}){\displaystyle\prod_{j=2}^{n+1}}a_{1,j}+&
\label{deltaS}
\end{eqnarray}
\begin{eqnarray*}
&(-1)^i{\displaystyle\sum_{i=1}^n}a_{i,i+1}f(\otimes\begin{pmatrix} 
1&a_{1,2}&...&a_{1,i-1}&a_{1,i}a_{1,i+1}&a_{1,i+2}&...&a_{1,n}&a_{1,n+1} \\
&1&...&a_{2,i-1}&a_{2,i}a_{2,i+1}&a_{2,i+2}&...&a_{2,n}&a_{2,n+1} \\
& &\ddots &\vdots&\vdots &\vdots&...&\vdots&\vdots \\
&& &1&a_{i-1,i}a_{i-1,i+1}&a_{i-1,i+2}&...&a_{i-1,n}&a_{i-1,n+1}\\
&& &&1&a_{i,i+2}a_{i+1,i+2}&...&a_{i,n}a_{i+1,n}&a_{i,n+1}a_{i+1,n+1} \\
&& &&&1&...&a_{i+2,n}&a_{i+2,n+1} \\
& & &&& &\ddots&\vdots&\vdots \\
&&&&& &&1&a_{n,n+1} \\
&&&&& && &1 \\
\end{pmatrix})+&
\end{eqnarray*}
\begin{eqnarray*}
&(-1)^{n+1}f(\otimes \begin{pmatrix} 1 & a_{1,2} & \cdots  &a_{1,n-1} & a_{1,n}\\ 
& 1  & \cdots &a_{2,n-1}& a_{2,n}\\
& &   \ddots&\vdots& \vdots\\ 
&  & & 1&a_{n-1,n}\\ 
&  & & &1 \end{pmatrix}){\displaystyle\prod_{j=1}^{n}}a_{j,n+1}.&
\end{eqnarray*}
The first and the last term in this sum are self explanatory. For the other terms, we collapse the $(n+1)\times (n+1)$ tensor  matrix along the $i$ and $i+1$ rows and columns to get a $n\times n$ matrix while the element $a_{i,i+1}$ becomes a coefficient in front of $f$. 

\subsection{Non-symmetric Operads}

The existence of a $G$-algebra structure on $H^*(A,A)$ was initially proved in \cite{g1}. Later the same result was obtained in \cite{gv} as a consequence of the homotopy $G$-algebra structure on $C^*(A,A)$. Using the same idea a similar result was proved for secondary Hochschild cohomology in \cite{s3}.  We recall here a few results and definitions concerning operads and $G$-algebras. For more details see \cite{mark}

 \begin{definition}
A non-symmetric unital operad is a sequence of $k$-vector spaces $\mathcal{P}=\{\mathcal{P}_n\}_{n\geq1}$ together with $k$-linear maps $$\circ_i:\mathcal{P}_n\otimes \mathcal{P}_m\to\mathcal{P}_{n+m-1},$$ for each $n,m\geq 1$, and $1\leq i\leq n$, and a distinguished element $\mathbf{1}\in \mathcal{P}_1$ such that the following relations hold for all $x\in \mathcal{P}_n, y\in \mathcal{P}_m$, and $z\in\mathcal{P}_k$:
\begin{gather}
   (x\circ_jz)\circ_iy=(x\circ_iy)\circ_{m+j-1}z,  \text{ if } 1\leq i<j\leq n\label{Odef1}\\ 
  (x\circ_iy)\circ_{i+j-1}z=x\circ_i(y\circ_jz),  \text{ if }  1\leq i \leq n \text{ and } 1\leq j \leq m\label{Odef2}\\
  x\circ_i\textbf{1} =x  \text{ if } 1\leq i \leq n\label{Odef3}\\
  \textbf{1} \circ_1x=x\label{Odef4}
 \end{gather}

\end{definition}

Given a non-symmetric operad  $\mathcal{P}=\{\mathcal{P}_n\}_{n\geq1}$,  one defines
$\circ:\mathcal{P}_n\otimes \mathcal{P}_m \to \mathcal{P}_{n+m-1}$
$$x\circ y=\sum_{i=1}^n(-1)^{(i-1)(m-1)}x\circ_i y.$$
One can show that $[x,y]=x\circ y-(-1)^{(n-1)(m-1)}y\circ x$ defines a graded Lie algebra on 
$\oplus_{i\geq 1} \mathcal{P}_i$ (see for example \cite{g1}).  For $x\in \mathcal{P}_n$, we set $|x|=n-1$.

A multiplication on an operad $\mathcal{P}$ is an element ${\mathfrak m}\in {\mathcal P}_2$ such that  ${\mathfrak m}\circ {\mathfrak m}=0$. The following was proved in \cite{gv}. 
\begin{theorem}
A multiplication on an operad ${\mathcal P}$ induces a homotopy $G$-algebra structure on 
$\oplus_{n\geq 1} {\mathcal P}_n$ where the cup product is 
\begin{eqnarray}\label{cdot}
x\cdot y:=(-1)^{|x|+1}{\mathfrak m}\{x,y\},
\end{eqnarray}  and the differential 
\begin{eqnarray}\label{diff} dx={\mathfrak m}\circ x-(-1)^{|x|}x\circ {\mathfrak m}. \end{eqnarray} One can see that $d^2=0$ and $\mathrm{deg}(d)=1$. \label{th1}
\end{theorem}
We will need the next two identities proved in \cite{gv}. 
\begin{eqnarray}
\begin{split}
[x,y\cdot z]-[x,y]\cdot z-(-1)^{|x|(|y|+1)}y\cdot[x,z]=(-1)^{|x|+|y|+1}(d(x\{y,z\})-\\
d(x)\{y,z\}-(-1)^{|x|}x\{dy,z\}-(-1)^{|x|+|y|}x\{y,dz\}), \label{jacobi}
\end{split}
\end{eqnarray}
\begin{eqnarray}
x\cdot y-(-1)^{(|x|+1)(|y|+1)}y\cdot x=(-1)^{|x|}(d(x\circ y)-dx\circ y-(-1)^{|x|}x\circ dy).\label{gradcom}
\end{eqnarray}

\section{A deformation theory associated to $H_{S^2}^*(A,A)$}

It is well know that Hochschild cohomology $H^*(A,A)$ gives information about deformations of the algebra structure on $A[[t]]$. Similar results are known for the Hochschild cohomology $H^*(A,M)$ \cite{h}, secondary cohomology \cite{sta}, cohomology of oriented algebras \cite{kp}, etc. 

In this section, we address a question asked by Bruce Corrigan regarding the importance of higher Hochschild cohomology in deformation theory. More precisely, we describe  a deformation theory that is controlled by $H_{S^2}^*(A,A)$.

Consider a  map $u:A[[t]]\to A[[t]]$ that is $k[[t]]$ linear and is determined by 
$$u(a)=a+u_1(a)t+u_2(a)t^2+...\in A[[t]],$$ 
where for every $i\geq 1$ we have that $u_i:A\to A$ is $k$-linear. Suppose that the map $u$ satisfies the identity  
\begin{eqnarray}
u(ab)u(c)=u(a)u(bc). \label{eq1}
\end{eqnarray}
Then we must have the following identity:
\begin{eqnarray*}
(ab)c+(abu_1(c)+u_1(ab)c)t+(abu_2(c)+u_1(ab)u_1(c)+u_2(ab)c)t^2+...=\\
a(bc)+(au_1(bc)+u_1(a)bc)t+(au_2(bc)+u_1(a)u_1(bc)+u_2(a)bc)t^2+....
\end{eqnarray*}
This means that in order for $u$ to satisfies equation (\ref{eq1}) $mod\; t^2$ we must have $abu_1(c)-au_1(bc)+cu_1(ab)-u_1(a)bc=0$ or equivalently $\delta_2^{S^2}(u_1)=0$. Also, in order to have (\ref{eq1}) $mod\; t^3$ we must have $abu_2(c)+u_1(ab)u_1(c)+u_2(ab)c=au_2(bc)+u_1(a)u_1(bc)+u_2(a)bc$, or equivalently 
$\delta_2^{S^2}(u_2)=u_1\circ u_1$.  Where for $f,g:A\to A$ we define $f\circ g:A^{\otimes 3}\to A$ determined by 
$$(f\circ g) \left(\displaystyle\otimes
\left(
\begin{array}{ccc}
1& a& b\\
& 1&c\\
&&1
\end{array} \right)
\right)=f(bc)g(a)-f(ab)g(c).$$
More generally,  for every $n\geq 1$ in order to have the identity (\ref{eq1}) $mod\;t^{n+1}$ the following identity must be satisfied 
$$\delta_2^{S^2}(u_n)=u_{n-1}\circ u_1+u_{n-2}\circ u_2+...+u_1\circ u_{n-1}.$$ 
To summarize, we have the following result
\begin{theorem} Let $A$ be a commutative $k$-algebra and $u:A[[t]]\to A[[t]]$ determined by 
$$u(a)=a+u_1(a)t+u_2(a)t^2+...\in A[[t]].$$ 
(i) If $u$ satisfies equation (\ref{eq1}) $mod\;t^{2}$  then $u_1\in Z^2_{S^2}(A,A)=H^2_{S^2}(A,A)$. \\
(ii) Suppose that $u$ satisfies equation (\ref{eq1}) $mod\;t^{n+1}$. Then we can extend $u$ so that it satisfies equation (\ref{eq1}) $mod\;t^{n+2}$ if and only if $$u_{n}\circ u_1+u_{n-1}\circ u_2+...+u_1\circ u_{n}=0\in H^3_{S^2}(A,A).$$
\end{theorem}
\begin{proof}
It follows from the above discussion. 
\end{proof}

\begin{remark} Notice that if $u:A\to A$ is a $k$ linear map that satisfies the identity $u(ab)u(c)=u(a)u(bc)$ and $u(1)\in U(A)$ then $v:A\to A$ $v(a)=u(a)u(1)^{-1}$ is a morphism of $k$-algebras.
\end{remark} 

\begin{example}
Take $u:k[X][[t]]\to k[X][[t]]$ to be an algebra morphism determined by $u(X)=X+tP(X)$ where $P(X)\in k[X]$. One can see that the corresponding two cocycle $u_1\in H^2_{S^2}(k[X],k[X])\subseteq Hom_k(k[X],k[X])$ is given by  $$u_1(A(X))=\frac{\partial A }{\partial X}(X)P(X).$$ 
\end{example}

\begin{example}
Take $u:k[X,Y][[t]]\to k[X,Y][[t]]$ to be an algebra morphism determined by $u(X)=X+tP(X,Y)$ and $u(Y)=Y+tQ(X,Y)$, where $P(X,Y)$ and $Q(X,Y)\in k[X,Y]$. One can see that the corresponding two cocycle $u_1\in H^2_{S^2}(k[X,Y],k[X,Y])\subseteq Hom_k(k[X,Y],k[X,Y])$ is given by  $$u_1(A(X,Y))=\frac{\partial A }{\partial X}(X,Y)P(X,Y)+\frac{\partial A}{\partial Y}(X,Y)Q(X,Y).$$ 
\end{example}

\section{G-algebra structure on $H_{S^2}^*(A,A)$}

In this section, we show the existence of a non-symmetric operad structure on $C^*_{S^2}(A,A)$.  This in turn gives a $G$-algebra structure at the level of higher order Hochschild cohomology over $S^2$.

We need to set a few notations.   Consider an element $T\in A^{\otimes \frac{s(s-1)}{2}}$ as in (\ref{matrixT}). It is convenient to split it into sub-tensors. Denote the $m$ by $n$ rectangular sub-tensor starting in position $(i,j)$ by
$$R_{i,j}^{m,n}=\otimes\begin{pmatrix}  a_{i,j} & a_{i,j+1}  & \cdots & a_{i,j+n-1}\\ a_{i+1,j} & a_{i+1,j+1} & \cdots & a_{i+1,j+n-1}\\ \vdots &\vdots & \ddots & \vdots\\ a_{i+m-1,j} & a_{i+m-1,j+1} & \cdots & a_{i+m-1,j+n-1} \end{pmatrix}.$$
 
Denote the triangular sub-tensor of dimension $m$ starting at position $(i,i)$ by $$T_i^m=\otimes\begin{pmatrix} 1 & a_{i,i+1} & a_{i,i+2}  & \cdots & a_{i,i+m-1}\\ & 1 & a_{i+1,i+3}  & \cdots & a_{i+1,i+m-1}\\ &  & \ddots & \dots & \vdots\\ & & &1& a_{i+m-2,i+m-1}\\ & & & & 1 \end{pmatrix}.$$ Notice that the sub tensors $R_{i,m}^{j,n}$ and $T_i^m$ do not make sense by themselves. They are still part of the initial tensor, this is just a convenient way to refer to certain positions in the tensor element.  


Next we define some operations on sub tensors.
\begin{definition}
Define ${\bf H}:A^{\otimes mn}\rightarrow A^{\otimes m}$ and ${\bf V}:A^{\otimes mn}\rightarrow A^{\otimes n}$ determined by  $${\bf H}(R_{i,j}^{m,n})=\otimes\begin{pmatrix} &\prod_{t=j}^{j+n-1} a_{i,t}& \\[6pt] &\prod_{t=j}^{j+n-1} a_{i+1,t}&\\[6pt] &\vdots & \\[6pt] &\prod_{t=j}^{j+n-1} a_{i+m-1,t}&\end{pmatrix},$$ and 
$${\bf V}(R_{i,j}^{m,n})=\otimes\begin{pmatrix} \prod_{t=i}^{i+m-1}a_{t,j}  & \prod_{t=i}^{i+m-1} a_{t,j+1} & \cdots & \prod_{t=i}^{i+m-1} a_{t,j+n-1}\end{pmatrix}.$$
\end{definition}
That is, ${\bf H}$ turns a rectangular sub-tensor into a column sub-tensor by multiplying along each row, and ${\bf V}$ turns a rectangular sub-tensor into a row sub-tensor by multiplying along each column.  Notice that applying ${\bf H}$ and ${\bf V}$ in succession in either order to a rectangle yields nothing but the product of all the entries of the rectangle.




\begin{definition}
Take $f\in C^n_{S^2}(A,A)$ and $g\in C^m_{S^2}(A,A)$.  We define $f\circ_ig\in C^{n+m-1}_{S^2}(A,A)$ determined by
\begin{equation}
 (f\circ_ig)\begin{pmatrix} T_1^{i-1} & R_{1,i}^{i-1,m} & R_{1,i+m}^{i-1,n-i} \\[6pt] 
 & T_i^{m} & R_{i,i+m}^{m,n-i} \\[6pt] 
 & & T_{i+m}^{n-i}\label{circlei}\end{pmatrix}=
f\begin{pmatrix} T_1^{i-1} & {\bf H}(R_{1,i}^{i-1,m}) & R_{1,i+m}^{i-1,n-i} \\[6pt] 
& 1 & {\bf V}(R_{i,i+m}^{m,n-i}) \\[6pt] 
& & T_{i+m}^{n-i}\end{pmatrix} 
g(T_i^{m}).
\end{equation}
\end{definition}


\begin{lemma}
Let $C^n_{S^2}(A,A)=Hom_k(A^{\otimes \frac{n(n-1)}{2}},A)$ for $n\geq 1$. Take $\mathbf{1}\in C^1_{S^2}(A,A)=Hom_k(k,A)$ to be the map $\alpha\mapsto \alpha \cdot1_A$ for $\alpha\in k$. Then the maps $\circ_i$ from (\ref{circlei})  define a non-symmetric unital operad structure on $\{C^n_{S^2}(A,A)\}_{n\geq1}$.  
\end{lemma}
\begin{proof}
The identities (\ref{Odef3}) and (\ref{Odef4}) involving the distinguished element $\mathbf{1}$ are straightforward. Indeed, for $f\in C^n_{S^2}(A,A)$ we have
$$(\mathbf{1}\circ_1f)(\otimes\begin{pmatrix} 
1 & a_{1,2} & a_{1,3}  & \cdots & a_{1,n}\\ 
& 1 & a_{2,2}  & \cdots & a_{2,n}\\ 
&  & \ddots & & \vdots\\ 
& & &1& a_{n-1,n}\\ 
& & & & 1 \end{pmatrix})=\mathbf{1}(1)f(\otimes\begin{pmatrix} 
1 & a_{1,2} & a_{1,3}  & \cdots & a_{1,n}\\ 
& 1 & a_{2,2}  & \cdots & a_{2,n}\\ 
&  & \ddots & & \vdots\\ & & &1& a_{n-1,n}\\ 
& & & & 1 \end{pmatrix})=$$
$$ f(\otimes\begin{pmatrix} 
1 & a_{1,2} & a_{1,3}  & \cdots & a_{1,n}\\ 
& 1 & a_{2,2}  & \cdots & a_{2,n}\\ 
&  & \ddots & & \vdots\\ 
& & &1& a_{n-1,n}\\ 
& & & & 1 \end{pmatrix}),$$ so $\mathbf{1}\circ_1f=f$

The identity (\ref{Odef4}) can be shown in a similar manner.


Next, take $f\in C^n_{S^2}(A,A)$, $g\in C^m_{S^2}(A,A)$, $h\in C^p_{S^2}(A,A)$, and let  $1\leq i<j\leq n.$ We want to check identity (\ref{Odef1}), or $(f\circ_jh)\circ_ig=(f\circ_ig)\circ_{m+j-1}h\in C^{n+m+p-2}_{S^2}(A,A)$. Take a general element $$T=\otimes\begin{pmatrix} 1 & a_{1,2} & a_{1,3}  & \cdots & a_{1,n+m+p-2}\\ 
& 1 & a_{2,2}  & \cdots & a_{2,n+m+p-2}\\
&  & \ddots & & \vdots\\ 
& & &1& a_{n+m+p-3,n+m+p-2}\\ 
& & & & 1 \end{pmatrix},$$
and write it as  
$$T=\begin{pmatrix} T_1^{i-1} & R_{1,i}^{i-1,m} & R_{1,i+m}^{i-1,j-i-1}& R_{1,m+j-1}^{i-1,p}& R_{1,m+j+p-1}^{i-1,n-j} \\[6pt] 
& T_i^m & R_{i,i+m}^{m,j-i-1}& R_{i,m+j-1}^{m,p} & R_{i,m+j+p-1}^{m,n-j}\\[6pt] 
& & T_{i+m}^{j-i-1} & R_{i+m,m+j-1}^{j-i-1,p} & R_{i+m,m+j+p-1 }^{j-i-1,n-j}\\[6pt]
& & & T_{m+j-1}^p & R_{m+j-1,m+j+p-1}^{p,n-j }\\[6pt] 
& & & & T_{m+j+p-1}^{n-j}\end{pmatrix}.$$ 

Since $i<j$, 
$$((f\circ_jh)\circ_ig)(T)=$$
 $$(f\circ_jh)\begin{pmatrix} T_1^{i-1} & {\bf H}(R_{1,i}^{i-1,m}) & R_{1,i+m}^{i-1,j-i-1}& R_{1,m+j-1}^{i-1,p}& R_{1,m+j+p-1}^{i-1,n-j} \\[6pt] 
 & 1 & {\bf V}(R_{i,i+m}^{m,j-i-1}) & {\bf V}(R_{i,m+j-1}^{m,p}) & {\bf V}(R_{i,m+j+p-1}^{m,n-j})\\[6pt] & & T_{i+m}^{j-i-1} & R_{i+m,m+j-1}^{j-i-1,p} & R_{i+m,m+j+p-1 }^{j-i-1,n-j}\\[6pt]
 & & & T_{m+j-1}^p & R_{m+j-1,m+j+p-1}^{p,n-j }\\[6pt] 
 & & & & T_{m+j+p-1}^{n-j}\end{pmatrix} g(T_i^{m})=$$
$$f\begin{pmatrix} T_1^{i-1} & {\bf H}(R_{1,i}^{i-1,m}) & R_{1,i+m}^{i-1,j-i-1}& {\bf H}(R_{1,m+j-1}^{i-1,p})& R_{1,m+j+p-1}^{i-1,n-j} \\[6pt] 
& 1 & {\bf V}(R_{i,i+m}^{m,j-i-1}) & {\bf H}({\bf V}(R_{i,m+j-1}^{m,p})) & {\bf V}(R_{i,m+j+p-1}^{m,n-j})\\[6pt] 
& & T_{i+m}^{j-i-1} & {\bf H}(R_{i+m,m+j-1}^{j-i-1,p}) & R_{i+m,m+j+p-1 }^{j-i-1,n-j}\\[6pt]
& & & 1 & {\bf V}(R_{m+j-1,m+j+p-1}^{p,n-j})\\[6pt] 
& & & & T_{m+j+p-1}^{n-j}\end{pmatrix} h(T_{m+j-1}^p) g(T_i^{m}).$$

On the other hand, $$((f\circ_ig)\circ_{m+j-1}h)(T)=$$
$$(f\circ_ig)\begin{pmatrix} T_1^{i-1} & R_{1,i}^{i-1,m} & R_{1,i+m}^{i-1,j-i-1} & {\bf H}(R_{1,m+j-1}^{i-1,p}) & R_{1,m+j+p-1}^{i-1,n-j} \\[6pt] 
& T_i^m & R_{i,i+m}^{m,j-i-1}& {\bf H}(R_{i,m+j-1}^{m,p}) & R_{i,m+j+p-1}^{m,n-j}\\[6pt] 
& & T_{i+m}^{j-i-1} & {\bf H}(R_{i+m,m+j-1}^{j-i-1,p}) & R_{i+m,m+j+p-1 }^{j-i-1,n-j}\\[6pt]
& & & 1 & {\bf V}(R_{m+j-1,m+j+p-1}^{p,n-j})\\[6pt] 
& & & & T_{m+j+p-1}^{n-j}\end{pmatrix}h(T_{m+j-1}^p)=$$
$$f\begin{pmatrix} T_1^{i-1} & {\bf H}(R_{1,i}^{i-1,m}) & R_{1,i+m}^{i-1,j-i-1} & {\bf H}(R_{1,m+j-1}^{i-1,p}) & R_{1,m+j+p-1}^{i-1,n-j} \\[6pt] 
& 1 & {\bf V}(R_{i,i+m}^{m,j-i-1}) & {\bf V}({\bf H}(R_{i,m+j-1}^{m,p})) & {\bf V}(R_{i,m+j+p-1}^{m,n-j})\\[6pt] 
& & T_{i+m}^{j-i-1} & {\bf H}(R_{i+m,m+j-1}^{j-i-1,p}) & R_{i+m,m+j+p-1 }^{j-i-1,n-j}\\[6pt]
& & & 1 & {\bf V}(R_{m+j-1,m+j+p-1}^{p,n-j})\\[6pt] 
& & & & T_{m+j+p-1}^{n-j}\end{pmatrix} g(T_i^m) h(T_{m+j-1}^p)$$
Since $A$ is commutative, we are done. 

Next we check identity (\ref{Odef2}), i.e. $(f\circ_ig)\circ_{i+j-1}h=f\circ_i(g\circ_jh),\text{ for }1\leq i\leq n \text{ and } 1\leq j \leq m.$ Now write a general element as 
$$T=\begin{pmatrix} T_1^{i-1} & R_{1,i}^{i-1,j-1} & R_{1,i+j-1}^{i-1,p}& R_{1,i+j+p-1}^{i-1,m-j}& R_{1,i+m+p-1}^{i-1,n-i} \\[6pt] 
& T_i^{j-1} & R_{i,i+j-1}^{j-1,p}& R_{i,i+j+p-1}^{j-1,m-j} & R_{i,i+m+p-1}^{j-1,n-i}\\[6pt] 
& & T_{i+j-1}^p & R_{i+j-1,i+j+p-1}^{p,m-j} & R_{i+j-1,i+m+p-1}^{p,n-i}\\[6pt]
& & & T_{i+j+p-1}^{m-j} & R_{i+j+p-1,i+m+p-1}^{m-j,n-i }\\[6pt] 
& & & & T_{i+m+p-1}^{n-i}\end{pmatrix}.$$ 

Then $$((f\circ_ig)\circ_{i+j-1}h)(T)=$$
 $$(f\circ_ig)\begin{pmatrix} T_1^{i-1} & R_{1,i}^{i-1,j-1} & {\bf H}(R_{1,i+j-1}^{i-1,p}) & R_{1,i+j+p-1}^{i-1,m-j}& R_{1,i+m+p-1}^{i-1,n-i} \\[6pt] 
 & T_i^{j-1} & {\bf H}(R_{i,i+j-1}^{j-1,p}) & R_{i,i+j+p-1}^{j-1,m-j} & R_{i,i+m+p-1}^{j-1,n-i}\\[6pt] 
 & & 1 & {\bf V}(R_{i+j-1,i+j+p-1}^{p,m-j}) & {\bf V}(R_{i+j-1,i+m+p-1}^{p,n-i})\\[6pt]
 & & & T_{i+j+p-1}^{m-j} & R_{i+j+p-1,i+m+p-1}^{m-j,n-i }\\[6pt] 
 & & & & T_{i+m+p-1}^{n-i}\end{pmatrix} h(T_{i+j-1}^p)=$$
$$f\begin{pmatrix} T_1^{i-1} & {\bf H}(R_{1,i}^{i-1,m+p-1}) & R_{1,i+m+p-1}^{i-1,n-i} \\[6pt] 
& 1 & {\bf V}(R_{i,i+m+p-1}^{m+p-1,n-i}) \\[6pt] 
& & T_{i+m+p-1}^{n-i}\end{pmatrix} g\begin{pmatrix} T_i^{j-1} & {\bf H}(R_{i,i+j-1}^{j-1,p}) & R_{i,i+j+p-1}^{j-1,m-j} \\[6pt] 
& 1 & {\bf V}(R_{i+j-1,i+j+p-1}^{p,m-j}) \\[6pt] 
& & T_{i+j+p-1}^{m-j}\end{pmatrix} h(T_{i+j-1}^p).$$

On the other hand, $$(f\circ_i(g\circ_jh))(T)=$$ 
$$f\begin{pmatrix} T_1^{i-1} & {\bf H}(R_{1,i}^{i-1,m+p-1}) & R_{1,i+m+p-1}^{i-1,n-i} \\[6pt] 
& 1 & {\bf V}(R_{i,i+m+p-1}^{m+p-1,n-i}) \\[6pt] 
& & T_{i+m+p-1}^{n-i}\end{pmatrix}(g\circ_jh)(T_i^{m+p-1})=$$
$$f\begin{pmatrix} T_1^{i-1} & {\bf H}(R_{1,i}^{i-1,m+p-1}) & R_{1,i+m+p-1}^{i-1,n-i} \\[6pt] 
& 1 & {\bf V}(R_{i,i+m+p-1}^{m+p-1,n-i}) \\[6pt] 
& & T_{i+m+p-1}^{n-i}\end{pmatrix}(g\circ_jh)
\begin{pmatrix} T_i^{j-1} & R_{i,i+j-1}^{j-1,p} & R_{i,i+j+p-1}^{j-1,m-j} \\[6pt] 
& T_{i+j-1}^{p} & R_{i+j-1,i+j+p-1}^{p,m-j} \\[6pt] 
& & T_{i+j+p-1}^{m-j}\end{pmatrix}=$$
$$f\begin{pmatrix} T_1^{i-1} & {\bf H}(R_{1,i}^{i-1,m+p-1}) & R_{1,i+m+p-1}^{i-1,n-i} \\[6pt] 
& 1 & {\bf V}(R_{i,i+m+p-1}^{m+p-1,n-i}) \\[6pt] 
& & T_{i+m+p-1}^{n-i}\end{pmatrix}
g\begin{pmatrix} T_i^{j-1} & {\bf H}(R_{i,i+j-1}^{j-1,p}) & R_{i,i+j+p-1}^{j-1,m-j} \\[6pt] 
& 1 & {\bf V}(R_{i+j-1,i+j+p-1}^{p,m-j}) \\[6pt] 
& & T_{i+j+p-1}^{m-j}\end{pmatrix}  h(T_{i+j-1}^p),$$
which completes the proof.
\end{proof}


\begin{remark}
Let $A$ be a commutative algebra.  For $f\in C^n_{S^2}(A,A)$ and $g\in C^m_{S^2}(A,A)$, set $f\circ g=\sum\limits_{i=1}^{n}(-1)^{(i-1)(m-1)}f\circ_ig$.  Then the bracket 
$$[f,g]=f\circ g- (-1)^{(n-1)(m-1)}g\circ f,$$ 
defines a graded Lie algebra on $\bigoplus\limits_{n\geq1} C^n_{S^2}(A,A)$ (see \cite{g1}).
\end{remark}

\begin{lemma}
The operad $\{C^n_{S^2}(A,A)\}_{n\geq 1}$ is multiplicative with the map ${\mathfrak m}\in C^2_{S^2}(A,A)$ defined by ${\mathfrak m}\begin{pmatrix}1 & a\\ & 1 \end{pmatrix}=a.$  In particular, we have a homotopy G-algebra structure on $\bigoplus\limits_{n\geq1} C^n_{S^2}(A,A)$.
\end{lemma}

\begin{proof}
We have
\begin{eqnarray*} 
{\mathfrak m}\circ{\mathfrak m}(\otimes \begin{pmatrix}1 & a&b\\
&1&c\\ 
&& 1 \end{pmatrix})&=&{\mathfrak m}(\otimes \begin{pmatrix}1 & bc\\
&1\\ \end{pmatrix}){\mathfrak m}(\otimes \begin{pmatrix}1 & a\\
&1\\ \end{pmatrix})-{\mathfrak m}(\otimes \begin{pmatrix}1 & ab\\
&1\\ \end{pmatrix}){\mathfrak m}(\otimes \begin{pmatrix}1 & c\\
&1\\ \end{pmatrix})\\
&=&(bc)a-(ab)c=0\\
\end{eqnarray*}
The second part follows from Theorem \ref{th1} (or \cite{gv}). 
\end{proof}

\begin{theorem}
For an commutative algebra $A$, we have a $G$-algebra structure on the higher order Hochschild cohomology of $A$ over $S^2$.
\end{theorem}
\begin{proof}
From Theorem \ref{th1}, we know that $\mathfrak{m}$ gives a differential $d$ on $\{C_{S^2}^n(A,A)\}_{n\geq1}$, namely $$d(f)=\mathfrak{m}\circ f - (-1)^{n-1}f\circ\mathfrak{m}.$$ We also have the differential $\delta^{S^2}_n$ that defines the higher order Hochschild cohomology (see (\ref{deltaS})). We want to show $(-1)^{n-1}d(f)=\delta^{S^2}_n(f)$. 


Take $T\in A^{\otimes\frac{n(n-1)}{2}}$. First note that,
$$(\mathfrak{m}\circ f)(T)=$$
$$(-1)^{(1-1)(n-1)}\mathfrak{m}
\begin{pmatrix}  & 1 & {\bf V}(R_{1,1+n}^{n,1}) \\[6pt] 
& & 1\end{pmatrix}
f(T_1^{n})+ (-1)^{(2-1)(n-1)}\mathfrak{m}
\begin{pmatrix} 1 & {\bf H}(R_{1,2}^{1,n})  \\[6pt] 
& 1 \end{pmatrix} f(T_2^{n})=$$ 
$$\prod_{i=1}^{n}a_{i,n+1}f(T_1^n)+(-1)^{n-1}\prod_{i=2}^{n+1} a_{1,i}f(T_2^n).$$

Also, $$-(f\circ\mathfrak{m})(T)=-\sum\limits_{i=1}^{n}(-1)^{(i-1)(2-1)}(f\circ_i\mathfrak{m})(T)=$$
 $$(-1)^{i}\sum\limits_{i=1}^{n}f\begin{pmatrix} T_1^{i-1} & {\bf H}(R_{1,i}^{i-1,2}) & R_{1,i+2}^{i-1,n-i} \\[6pt] 
 & 1 & {\bf V}(R_{i,i+2}^{2,n-i}) \\[6pt] 
 & & T_{i+2}^{n-i}\end{pmatrix}\mathfrak{m}(T_i^{2})=$$ 
$$ (-1)^{i} \sum\limits_{i=1}^{n}a_{i,i+1} f\begin{pmatrix} T_1^{i-1} & {\bf H}(R_{1,i}^{i-1,2}) & R_{1,i+2}^{i-1,n-i} \\[6pt] 
& 1 & {\bf V}(R_{i,i+2}^{m,n-i}) \\[6pt] 
& & T_{i+2}^{n-i}\end{pmatrix}.$$

And so

$$(-1)^{n-1}d(f)(T)=(-1)^{n-1}((\mathfrak{m}\circ f-(-1)^{n-1}f\circ\mathfrak{m})(T))=$$
$$(-1)^{n-1}\prod_{i=1}^{n}a_{i,n+1} f(T_1^n)+\prod_{i=2}^{n+1} a_{1,i} f(T_2^n)+$$
$$(-1)^{i} \sum\limits_{i=1}^{n}a_{i,i+1} f\begin{pmatrix} T_1^{i-1} & {\bf H}(R_{1,i}^{i-1,2}) & R_{1,i+2}^{i-1,n-i} \\[6pt] 
& 1 & {\bf V}(R_{i,i+2}^{m,n-i}) \\[6pt] 
& & T_{i+2}^{n-i}\end{pmatrix}=\delta^{S^2}_n(f)(T).$$

Since the sign of the differential doesn't matter in homology, the $G$-algebra structure on the higher order Hochschild cohomology of $A$ over $S^2$ follows from the homotopy $G$-algebra structure already established on $\{C_{S^2}^n(A,A)\}_{n\geq 1}$.  For example Equation  (\ref{jacobi}) shows the graded Lebiniz compatibility,  while Equation (\ref{gradcom}) demonstrates that the cup product is graded commutative. One can easily check the associativity of the cup product. Finally the Jacobi identity follows from \cite{g1}. \end{proof}

\begin{remark}  Ginot introduced in \cite{gi} a $d+1$-algebra structure on $H_{S^d}^*(A,A)$. Essentially this means that there is a cup product and a bracket of degree $d$ on $H_{S^d}^*(A,A)$. In the case of the sphere $S^2$ this means that the bracket is of degree $2$ (i.e.  $[.,.]_3:H_{S^2}^m(A,A)\otimes H_{S^2}^n(A,A)\to H_{S^2}^{m+n-2}(A,A)$), which obviously is not the case in this paper. Moreover it is not clear if there is any relation with our results. 
\end{remark}

\begin{remark} The results in this paper are in the spirit of those obtained in \cite{sta} and \cite{s3}, however they are not particular cases of those constructions. 
\end{remark}

\bibliographystyle{amsalpha}

\end{document}